\documentclass[12pt]{article}

\usepackage{amssymb}
\usepackage{amsmath}
\usepackage{amsthm}
\usepackage{anysize}
\usepackage{bookmark}

\marginsize{1.125in}{1.125in}{0.5in}{1.25in}

\newtheorem{theorem}{Theorem}[section]
\newtheorem{lemma}[theorem]{Lemma}
\newtheorem{proposition}[theorem]{Proposition}
\newtheorem{remark}[theorem]{Remark}

\usepackage{hyperref}
\hypersetup{
colorlinks=true,
allcolors=black
}

\date{}
\title{A stochastic comparison result for the multitype contact process with unequal death rates}
\author{Joseph P. Stover\footnote{Department of Mathematics, Gonzaga University, 502 E Boone Ave, Spokane, WA 99258. E-mail:  \href{mailto:stover@gonzaga.edu}{stover@gonzaga.edu}}}

\begin{document}

\maketitle

{\let\thefootnote\relax\footnotetext{Keywords: stochastic order; multitype contact process; interacting particle system; attractive}}

{\let\thefootnote\relax\footnotetext{Published in: Statistics and Probability Letters 162 (2020) 108763}}
{\let\thefootnote\relax\footnotetext{DOI: \href{https://doi.org/10.1016/j.spl.2020.108763}{10.1016/j.spl.2020.108763}}

\begin{abstract}\addcontentsline{toc}{section}{Abstract}
A stochastic comparison result that makes progress towards understanding the classical multitype contact process with unequal death rates is given. It has long been conjectured that the particle type with the largest birth to death rate ratio survives and the other dies out. A point process coupling result of Broman \cite{broman2007} is used to give a sufficient condition for when the dominant particle type survives.
\end{abstract}

\section{Introduction}

The contact process is a continuous-time stochastic process that models a spreading organism or infection over a discrete graph. For a comprehensive historical and theoretical background, see Liggett's books \cite{LiggettIPS} and \cite{Liggett99}. Here we consider the nearest neighbor contact process on the $d$-dimensional square lattice. For site $x\in \mathbb Z^d$ its set of nearest neighbors is denoted $\mathcal N(x)$ and includes all $y$ that are $\ell_1$ distance of $1$ away. For configuration $\eta\in\{0,1\}^{\mathbb Z^d}$ at each site $x\in\mathbb Z^d$ the flip rates for $\eta_t(x)$ the configuration of the process at time $t$ are given by 
\begin{equation*}
\begin{tabular}{cc}
Transition & Rate\\
$0\rightarrow1$ & $\lambda n_1$\\
$1\rightarrow0$ & 1\\ 
\end{tabular}
\end{equation*}
where $n_1=\sum_{y\in \mathcal N(x)}\eta_t(y)$ is the number of neighboring sites that are in state $1$. We will refer to sites in state 1 as infected or occupied by type 1 and state 0 as empty. There is a critical birth rate parameter value $\lambda_c$ such that when $\lambda>\lambda_c$ the process is said to survive, and otherwise it dies out. The precise numerical value of $\lambda_c=\lambda_c(d)$ depends on the dimension of the lattice. The contact process is called supercritical when $\lambda>\lambda_c$.

Informally, survival means that when the initial state has finitely many infected sites there are infected sites at all future times. Let $A_t$ be the set of infected sites at time $t$ and let $\mathbb P^A$ be the distribution of the process with $A$ the set of initially infected sites and all other sites empty. The basic contact process exhibits what is called \textit{strong survival} defined by 
\begin{equation*}
\mathbb P^{\{x\}}(x\in A_t \text{ i.o.})>0
\end{equation*}
which means that when site $x$ is the only initially infected site, then there is a positive probability of it being infected infinitely often. When there are infinitely many initially infected sites, the distribution of the process converges weakly to the upper invariant measure $\overline\nu$ which is the limiting distribution for the process starting with all sites infected. Note that $\overline\nu$ is translation invariant as well. See \cite{Liggett99} for details on weak convergence and survival of the contact process.

When an interacting particle system satisfies a stochastic ordering relationship, it is often called \textit{attractive} or \textit{monotone}. The contact process possesses this property in the sense that if $\eta_t$ and $\tilde\eta_t$ are two contact processes with birth rates $\lambda\leq\tilde\lambda$, respectively, and initial conditions satisfying $\eta \leq \tilde\eta$, then these two processes can be coupled together so that $\eta_t \leq \tilde\eta_t$ for all $t$. Here $\eta \leq \tilde\eta$ means that $\eta(x) \leq \tilde\eta(x)$ for all $x\in\mathbb Z^d$. These partial ordering relationships can also be stated as
\begin{align*}
\{x \mid \eta_t(x)=0\} &\supset \{x \mid \tilde\eta_t(x)=0\}\\
\{x \mid \eta_t(x)=1\} &\subset \{x \mid \tilde\eta_t(x)=1\}.
\end{align*}
Thus we say that the contact process is attractive and monotone increasing in $\lambda$. 

The graphical construction for the contact process was originally introduced by Harris \cite{harris1978} and has proved to be quite a useful and presentable tool over the years. For each site $x$, there is a vertical timeline starting at $t=0$ with $t$ increasing in the upward direction. On each timeline, place $\times$ symbols at intensity $1$ (death marks) and for each neighboring site $y$ place arrows from $y$ to $x$ at intensity $\lambda$. Site $x$ can be infected (if it is empty) by site $y$ (if it is infected) along these arrows. An infection at $x$ dies when it hits a death mark. One can envision fluid placed at the start of the timelines at $t=0$ for sites which are to be initially infected. The fluid flows up the graphical structure stopping at $\times$'s and branching at arrows according to the direction they are pointed. An active path from $(x_1,t_1)$ to $(x_2,t_2)$ is a path in the graphical structure that connects site $x_1$ at time $t_1$ to site $x_2$ at time $t_2$ following vertical timelines, branching from one timeline to another along arrows, and never crossing any death marks. 

The distribution of the contact process is completely characterized by the distribution of the graphical structure. Given initial state $\eta$, the probability that $\eta_t$ satisfies some criteria is equivalent to the probability that the graphical structure satisfies that criteria using $\eta$ as the initial state and following all active paths from initially infected sites to determine the configuration at time $t$. For example, if site $x$ is the only initially infected site, then the probability that the process survives strongly is equivalent to the probability that the graphical structure admits active paths from $(x,0)$ to $(x,t)$ for infinitely many $t$. Again, see Liggett \cite{Liggett99} for more details. The arguments in this paper are based on the use of a graphical construction for establishing stochastic comparison relationships between processes. 

Broman \cite{broman2007} proved a useful stochastic domination result for a point process that randomly switches between two arrival rates and used this to study the contact process in a randomly evolving environment (CPREE) where the death rate randomly flips between two values. Remenik \cite{remenik2008} used Broman's result to study a CPREE-type model where sites become blocked randomly (similar to a contact process where the birth rate switches between two values). 

The multitype contact process (MCP) was originally studied by Neuhauser \cite{Neuhauser1992} 
and is constructed from two contact processes competing on the same lattice (with particle types 1 and 2 and empty sites 0). It has long been conjectured for the multitype contact process that the type with the greater birth to death rate ratio (BDR) will exclude the other type. We refer to this as the BDR conjecture. We use Broman's coupling result to give a sufficient condition on when one type survives and excludes the other.

\section{Multitype contact process}

The transition rates for the nearest neighbor MCP at each site $x\in\mathbb Z^d$ are 
\begin{equation*}
\begin{tabular}{cl}
Transition & Rate \\
$0\rightarrow1$ & $\beta_1 n_1$ \\
$1\rightarrow0$ & $\delta_1$ \\
$0\rightarrow2$ & $\beta_2 n_2$\\
$2\rightarrow0$ & $\delta_2$\\
\end{tabular}
\end{equation*}
where $n_i$ is the number of nearest neighbors of type $i$. We call $\beta_i$ and $\delta_i$ birth and death rates for type $i$. We say type $i$ is supercritical if its BDR $\beta_i/\delta_i$ is greater than $\lambda_c$. A supercritical type survives in absence of other type since it behaves just like a standard supercritical contact process. 

Neuhauser (see \cite{Neuhauser1992} Theorem 1) showed that when $\delta_1=\delta_2=1$, then the particle type with the greater birth rate survives (if it is supercritical) and the other dies out. It was also conjectured more generally that when the death rates are unequal, the type with the greater BDR survives (assuming it is supercritical) and excludes the other type. The type of survival Neuhauser proved is as follows: when $\mu$ is a translation invariant initial distribution with $\mu(\eta(x)=2)>0$, then the distribution of $\eta_t$ converges weakly to $\overline\nu_2$ the upper invariant measure for the process started with all sites in state 2. Note that $\overline\nu_2$ is exactly the translation invariant upper stationary measure for a standard contact process (with particle types appropriately re-labeled) and satisfies $\overline\nu_2(\eta(x)=1)=0$, $\overline\nu_2(\eta(x)=2)>0$.

The MCP is attractive and monotone in each of its parameters. If we carefully re-order the particle types, e.g. replace the label of type $1$ by ``$-1$'' to create the ordering $-1<0<2$ for the possible states at each site, then we have the traditional stochastic ordering in that a coupling exists such that $\eta\leq\xi$ implies $\eta_t\leq\xi_t$ for all time (see \cite{stoverIMPS} and \cite{borrello2011}). In this case, type $-1$ is one of the competing species, type $2$ is the other, and $0$ still means empty. If $\xi_t$ has a higher birth rate and lower death rate for type 2 and lower birth rate and higher death rate for type 1, this ordering is also preserved giving monotonicity in each parameter. Stover \cite{stoverIMPS} noticed this fact and also developed a method for determining attractiveness for a broader class of multitype contact processes. Borrello \cite{borrello2011} developed a method to assess attractiveness when jumps from site to site are allowed in addition to births and deaths. 

\subsection{Graphical construction for the MCP}
The graphical construction that we use for the MCP is as follows. On the timeline for each $x$ independently place $\times$ symbols with intensity $\delta_2$ (where only type 2 is killed) and $\bullet$ symbols with intensity $\delta_1$ (where only type 1 is killed). These are referred to as death marks. For each neighboring site $y$, distribute with intensity $\beta_1$ arrows labeled with a 1 and the tip pointed towards $x$ and the other end attached to $y$. Similarly place arrows labeled with a 2 at intensity $\beta_2$. 
These symbols are referred to as 1-arrows and 2-arrows, respectively. Arrows indicate where a birth of the labeled type may occur at $x$ (if it is empty) from $y$ (if it is occupied by the labeled type). On each timeline, one could also choose to reverse the direction of the arrows (as long as arrows going both directions are generated at the correct intensities), but it is key for our purposes that they are generated at site $x$ (the site to be infected) from the Poisson process at the tips of the arrows. 

The attractiveness of the MCP can be seen directly from this graphical representation. When two processes $\eta_t$ and $\xi_t$ initially satisfy 
\begin{equation}\label{eq:mcp_setattr}
\begin{split}
\{x \mid \eta_t(x)=1\} &\supset \{x \mid \xi_t(x)=1\}\\
\{x \mid \eta_t(x)=0 \text{ or } 1\} &\supset \{x \mid \xi_t(x)=0 \text{ or } 1\}\\
\{x \mid \eta_t(x)=0 \text{ or } 2\} &\subset \{x \mid \xi_t(x)=0 \text{ or } 2\}\\
\{x \mid \eta_t(x)=2\} &\subset \{x \mid \xi_t(x)=2\}
\end{split}
\end{equation}
at $t=0$, then each symbol in the graphical construction preserves these relationships and thus they hold for all $t$. If we instead replace the label of type 1 by a $-1$, then we have monotonicity in the form of $\eta_t \leq \xi_t$. 
The arguments in this paper will usually only require the first or last lines of \eqref{eq:mcp_setattr}. 

Since the MCP is monotone in each its parameters (see \cite{stoverIMPS} and \cite{borrello2011}), we have the following result which is essentially identical to Theorem 1 in Neuhaser \cite{Neuhauser1992} except with $\delta_1>\delta_2$ instead of $\delta_1=\delta_2=1$. 

\begin{proposition}\label{lem:mcp1}
For the MCP with parameters $\beta_1<\beta_2$, $\delta_1>\delta_2$, with type 2 supercritical, type 2 survives and type 1 dies out in the sense that if the initial distribution $\mu$ is translation invariant with $\mu(\eta(x)=2)>0$, then the distribution of the process at time $t$ converges weakly to $\overline\nu_2$ the upper invariant measure starting from all sites initially in state 2.
\end{proposition}
\begin{proof}
Let $\eta_t$ be the MCP with parameters $\delta_1=\delta_2=1$ and $\beta_2>\beta_1$ and $\beta_2>\lambda_c$. Let $\xi_t$ have the same rates as $\eta_t$ except with $\delta_1=1+\sigma$ for $\sigma>0$. The MCP is monotone in parameter $\delta_1$ in the sense that if
\begin{align*}
\{x \mid \eta_t(x)=1\} &\supset \{x \mid \xi_t(x)=1\}\\
\{x \mid \eta_t(x)=2\} &\subset \{x \mid \xi_t(x)=2\}
\end{align*}
holds initially at $t=0$, then it holds for all $t$ when suitably coupled. To see this, consider the graphical representation described above with a slight modification. Distribute arrows as already declared, $\times$'s at rate $1$ but $\bullet$'s at rate $\sigma$. Both processes use arrows and $\times$'s identically, but only for $\xi_t$ does type 1 die at $\bullet$'s whereas all types die for both processes at $\times$ death marks. This graphical construction preserves the given subset comparison. 

Let $\mathbb P^\eta$, $\mathbb P^\xi$, and $\mathbb P^{\overline 2}$ be the distributions of the processes $\eta_t$, $\xi_t$ (both with translation invariant initial distribution $\mu$), and $\zeta_t$ the process started with all sites in state 2, respectively. 
These measures are stochastically ordered in the sense that their underlying processes can be coupled together so that if \eqref{eq:mcp_setattr} is satisfied at $t=0$ then it is for all $t>0$ (with an additional set on the right side for $\zeta_t$ appropriately included). Therefore we have that 
\begin{align*}
\mathbb P^\eta(\eta_t(x)=1) \geq \mathbb P^\xi(\xi_t(x)=1)  \geq \mathbb P^{\overline 2}(\zeta_t(x)=1) \\
\mathbb P^\eta(\eta_t(x)=2) \leq \mathbb P^\xi(\xi_t(x)=2)  \leq \mathbb P^{\overline 2}(\zeta_t(x)=2)
\end{align*}
for all $t\geq0$. Note that $\mathbb P^{\overline 2}(\zeta_t(x)=1)=0$.

Theorem 1 in Neuhauser \cite{Neuhauser1992} shows that $\mathbb P_t^\eta \Rightarrow \overline\nu_2$ (weak convergence of measures). By definition, $\mathbb P_t^{\overline 2} \Rightarrow \overline\nu_2$, thus $\mathbb P_t^\xi\Rightarrow \overline\nu_2$. So we have that type 2 survives since $\lim_{t\rightarrow\infty}\mathbb P (\xi_t(x)=2) = \overline\nu_2(\xi(x)=2)>0$ and type 1 dies out since $\lim_{t\rightarrow\infty}\mathbb P (\xi_t(x)=1) = \overline\nu_2(\xi(x)=1)=0$.
\end{proof}

The result of Proposition \ref{lem:mcp1} only requires that type 2 is supercritical, but it still applies and is more interesting when both types are supercritical. By a scaling argument, this generalizes to any parameters satisfying $\beta_1<\beta_2$, $\delta_1>\delta_2$ with $\frac{\beta_2}{\delta_2}>\lambda_c$.

We now present the main result. Let $B_t\subset\mathbb Z^d$ be the set of sites that are in state 2 at time $t$ and $\mathbb P^B$ be the distribution of the process with $B$ the set of sites that are initially occupied by type 2 and all other sites initially type 1.

\begin{theorem}[Type 2 strong survival sufficiency condition]\label{thm:cpree_mcp}
Consider the MCP with parameters $\beta_2=c\beta$, $\delta_2=1$, $\beta_1=\beta\alpha$, $\delta_1=\alpha$ with $\beta>\lambda_c$ and define
\begin{equation}\label{eq:mcp_lambar}
\overline{\lambda}(\beta,c,\alpha)=\frac12\left(c\beta+\alpha+2d\beta\alpha-\sqrt{(c\beta-\alpha-2d\beta\alpha)^2+8d\alpha c \beta^2}\right).
\end{equation}
If $\overline\lambda(\beta,c,\alpha)>\lambda_c$, then type 2 survives strongly in the sense that 
\begin{equation*}
\mathbb P^{\{x\}}(x\in B_t \text{ i.o.})>0.
\end{equation*}
\end{theorem}

This theorem is proved in Section \ref{sec:mainproof} and uses Broman's \cite{broman2007} point process coupling and a comparison to a standard contact process. 

This result does not give any useful information when $\overline\lambda(\beta,c,\alpha) \leq \lambda_c$. It is also important to note here that $\overline\lambda$ depends on the dimension of the lattice in addition to the other parameters.  
Taking derivatives of $\overline\lambda$ with respect to $c$ and $\alpha$ shows that $\overline\lambda$ is increasing in both $c$ and $\alpha$ (though it does require a bit of careful algebra and reasoning). 
We also have that 
$\overline{\lambda}(\beta,c,\alpha)\rightarrow \alpha$ as $c\rightarrow\infty$
thus so long as $\alpha>\lambda_c$ then $c$ can be chosen large enough so that type 2 survives. Also $\overline{\lambda}(\beta,c,\alpha)\rightarrow \frac{c\beta}{1+2d\beta}$ as $\alpha\rightarrow\infty$ so as long as $c\beta>(1+2d\beta)\lambda_c$ then $\alpha$ can be chosen large enough to make type $2$ survive. If $c\beta \leq \lambda_c$ then type 2 dies out trivially.  It is interesting in its own right that $\overline\lambda$ is increasing in $\alpha$ since this does not change the BDR for either type, it simply makes type 1 have faster dynamics so-to-speak. 

\subsection{Consistency with BDR conjecture}

Here we present a short argument to see that $c>1$ (giving type 2 the highest BDR) is required for $\overline\lambda>\lambda_c$. This is an important fact because if $c<1$ were allowed, then our result would be inconsistent with the conjecture that the type with the highest birth to death rate ratio survives with the other dying out.

A well-known lower bound on the standard contact process critical birth rate parameter value is given by $\lambda_c \geq \frac1{2d-1}$. We now assume all parameters except $c$ are fixed in such a way that there is a large enough value of $c$ which gives $\overline\lambda>\lambda_c$. Solving $\overline\lambda=\frac1{2d-1}$ for $c$ gives   
\begin{equation}\label{eq:cbnd}
c^*(\alpha,\beta,d)=
\frac{\alpha(1+2d\beta)(2d-1)-1}
{\beta(2d-1)\left(\alpha(2d-1)-1\right)}
\end{equation}
so that $c>c^*$ is necessary to make $\overline\lambda>\lambda_c$. We show that $c^*>1$. 
Since $\overline\lambda$ increases to $\alpha$ as $c\rightarrow\infty$, we need $\alpha>\lambda_c$ in order to make $\overline\lambda>\lambda_c$ possible for some sufficiently large value of $c$. Thus we can assume that $\alpha>\frac1{2d-1}$ which makes the denominator of \eqref{eq:cbnd} positive. Finally, some algebraic rearranging shows that $c^*>1$ if  
\begin{equation*}
\alpha-\frac1{2d-1}+\beta(\alpha+1)>0
\end{equation*}
which is true since $\beta(\alpha+1)>0$ and $\alpha>\frac1{2d-1}$. Therefore $c>1$ is required, and we conclude that Theorem \ref{thm:cpree_mcp} is consistent with the BDR conjecture. 

Also note that the requirement $\alpha>\lambda_c$ implies that, in $d=1$, type 1 necessarily has a higher death rate than type 2 for our result to hold. For $d=2$, the best known upper bound on $\lambda_c$ is $\frac2d$, thus apparently $\alpha>1$ may be required for our result to hold, but numerical simulations suggest that $\lambda_c<1$ in two dimensions. For $d\geq3$, $\alpha<1$ is allowed thus type 1 can have a lower death rate and slower overall dynamics, relative to type 2.

\section{Broman's result and a CPREE model}
In this section, we present some results that will be used to prove Theorem \ref{thm:cpree_mcp}. First, Broman's stochastic domination result for point processes is given (see \cite{broman2007} for full details). 

Consider $(B_t,X_t)$ a coupling of a two-state background process $B_t\in\{0,1\}$ and Poisson counting process $X_t$ whose arrival rate depends on the state of the background process. When $B_t=i$, $X_t$ has arrival rate $\alpha_i$, and $B_t$ flips between $0$ and $1$ independently of $X_t$. The transitions and rates are given as
\begin{equation*}
\begin{tabular}{cl}
Transition & Rate \\
$(0,k)\rightarrow(1,k)$ & $\gamma p$ \\
$(1,k)\rightarrow(0,k)$ & $\gamma (1-p)$ \\
$(0,k)\rightarrow(0,k+1)$ & $\alpha_0$ \\
$(1,k)\rightarrow(1,k+1)$ & $\alpha_1$ \\
\end{tabular}
\end{equation*}
It is possible for $X_t$ to stochastically dominate a standard Poisson counting process with rate $\lambda$. The following lemma is part of Theorem 1.4 from \cite{broman2007}. This result requires that the background process is initially at equilibrium with $B_0=1$ with probability $p$ and $B_0=0$ with probability $1-p$ initially.

\begin{lemma}[Broman's coupling]\label{lem:broman}
For $X_t$ defined above, let $\tilde X_t$ be a Poisson counting process with rate $\lambda$. Then $\lambda \leq \overline\lambda(\alpha_0,\alpha_1,\gamma,p)$ for
\begin{equation*}
\overline{\lambda}(\alpha_0,\alpha_1,\gamma,p)=\frac12\left(\alpha_1+\alpha_0+\gamma-\sqrt{(\alpha_1-\alpha_0-\gamma)^2+4\gamma(1-p)(\alpha_1-\alpha_0)}\right)
\end{equation*}
implies that $X_t$ and $\tilde X_t$ can be coupled together so that $\tilde X_t\leq X_t$.
Furthermore, $\overline\lambda$ is the maximum possible value of $\lambda$ where this coupling is possible.
\end{lemma}
Now we use this coupling result to establish a stochastic domination relationship between a point counting process coupled with the multitype contact process and Broman's two-rate point counting process.

\subsection{A CPREE model} \label{subsec:cpree}
Denote the MCP with parameters given in the statement of Theorem \ref{thm:cpree_mcp} by $\eta_t$. Type 1 in this MCP is now effectively turned into a dynamic randomly evolving environment by making its births spontaneous at the maximum rate by replacing $n_1$ by $2d$, the number of nearest neighbors on the $d$-dimensional square lattice. Denote this process by $\xi_t$ with transition rates given below. 
\begin{equation}\label{eq:cpree_rates}
\begin{tabular}{cl}
Transition & Rate \\
$0\rightarrow1$ & $2d\beta\alpha$ \\
$1\rightarrow0$ & $\alpha$ \\
$0\rightarrow2$ & $c\beta n_2$\\
$2\rightarrow0$ & $1$\\
\end{tabular}
\end{equation}
This is almost identical to the CPREE model studied by Remenik \cite{remenik2008} except that the presence of type 2 prevents the random environment from flipping to state 1 in our model. In Remenik's model, type 2 was killed when the the random environment flipped to state 1. Note that our type 1 is analogous to type $-1$ in Remenik's model, and our type 2 is analogous to type 1 there.

We have that the MCP $\eta_t$ stochastically dominates the CPREE $\xi_t$ in the sense that
\begin{align*}
\{x \mid \xi_t(x)=1\} &\supset \{x \mid \eta_t(x)=1\}\\
\{x \mid \xi_t(x)=2\} &\subset \{x \mid \eta_t(x)=2\}.
\end{align*}
In other words, since type 1 is born spontaneously in $\xi_t$ it can have more type 1's and $\eta_t$ can thus have more 2's. This is a straightforward result since both processes can be constructed on the same graphical construction that is given above for the MCP. 
Births for $\xi_t$ at tips of 1-arrows are spontaneous (if the site is empty), otherwise both processes behave identically at other graphical symbols. Each of these symbols preserves the subset comparison relationship given above.

\section{Proof of Theorem \ref{thm:cpree_mcp}}\label{sec:mainproof}

The following proposition is very similar to Proposition 3.2 in \cite{remenik2008}, but the proof is slightly different due to a differing construction which is needed for our particular application and that the underlying interacting particle systems are different. It is now important that our arrows have the tips pointed into the timeline they are attached to when generated from that timeline's point process.

\begin{proposition}\label{prop:cpree}
Let $\xi_t$ be the CPREE described above with transition rates given by \eqref{eq:cpree_rates} and $\tilde\xi_t$ be a standard contact process with birth rate $\lambda \leq \overline\lambda(\beta,c,\alpha)$ from \eqref{eq:mcp_lambar} and death rate $1$. These processes can be coupled so that if 
\begin{equation*}
\{x \mid \tilde\xi_t(x)=1\} \subset \{x \mid \xi_t(x)=2\}
\end{equation*}
holds initially for $t=0$, then it holds for all $t>0$.
\end{proposition}
\begin{proof}

Consider a single timeline in the MCP graphical construction for site $x$ and the graphical elements generated along it. Recall that arrows are generated with tips pointed towards $x$. If a 2-arrow from neighboring site $y$ to $x$ has a 1-arrow below it (with tip attached at $x$) with no death symbol for type 1 between them, we call the 2-arrow \textit{blocked} and otherwise \textit{unblocked}. A blocked 2-arrow may be unusable since births are spontaneous at the tips of 1-arrows.  Whether a blocked 2-arrow is truly unusable by type 2 depends on the precise history of the process up to that point. This is because flipping to type 1 at the tip of a 1-arrow (since births for type 1 are spontaneous) can be prevented by the site being in state 2. If the site is already in state 2, there may be a death mark for type 2 which causes the site to become vacant in time for the blocked 2-arrow to actually be usable. We are unconcerned with this. We choose to ignore blocked 2-arrows entirely in order to prevent the need to consider the type 1 population at all. 

For each neighbor $y$, let $Y_t$ be the counting process for unblocked 2-arrows from this neighbor to $x$. 
This $Y_t$ is identical to Broman's $X_t$ with $\alpha_0=0$, $\alpha_1=c\beta$, $\gamma=\alpha(1+2d\beta)$, and $p=(1+2d\beta)^{-1}$. We have $2d$ such neighbors and for each one a distinct $Y_t$ process. The background processes for these $Y_t$ are coupled together so that they become blocked and unblocked simultaneously (at rates $\alpha$ and $2d\beta\alpha$ respectively), but when unblocked they increment independently of one another. Note that these background processes must be started in equilibrium initially (still initially coupled together). This may seem like an undesirable requirement, but it doesn't affect the comparison since this just means the $Y_t$ processes are more easily dominated by the counting process for \textit{all} unblocked 2-arrows which starts with all sites unblocked. 

Since $\lambda \leq \overline\lambda$, Lemma \ref{lem:broman} shows that there is a Poisson counting process $\tilde X_t$ with rate $\lambda$ such that $\tilde X_t \leq Y_t$ for all $t\geq0$ (we require $2d$ such processes, one for each neighboring $y$). Identify the locations of the points for counting process $\tilde X_t$ with the locations of birth arrows (involving the same sites $x$ and $y$) for a standard contact process $\tilde\xi_t$ with birth rate $\lambda$ and death rate $1$ (and uses the $\times$ death marks). The counting process $Y_t$ places an unblocked 2-arrow at every location associated with a point counted by $\tilde X_t$, and additionally could possibly place more unblocked 2-arrows at other locations. The 2-arrows associated with $Y_t$ are exactly those unblocked ones of the CPREE $\xi_t$ under consideration. Again, the CPREE may have many more blocked 2-arrows that are ultimately usable for particular initial conditions. 

By construction, all arrows associated with $\tilde X_t$ can be traversed by type 2 in the CPREE and never blocked from below by a pre-existing type 1 infection. Those arrows can also be traversed by the standard contact process $\tilde\xi_t$. So type 2 in $\xi_t$ can traverse any arrow that type 1 in $\tilde\xi_t$ can (and possibly other arrows), and both types are only killed at $\times$'s. 
Thus we see that when
\begin{equation*}
\{x \mid \tilde\xi_t(x)=1\} \subset \{x \mid \xi_t(x)=2\}
\end{equation*}
holds initially at $t=0$, then it holds for all time.
\end{proof}

Now the proof of Theorem \ref{thm:cpree_mcp} is straightforward. 

\begin{proof}[Proof of Theorem \ref{thm:cpree_mcp}]
By putting together Section \ref{subsec:cpree} and Proposition \ref{prop:cpree}, we have that the MCP $\eta_t$ stochastically dominates the CPREE $\xi_t$ which stochastically dominates a standard contact process $\tilde\xi_t$ in the sense that if 
\begin{align}\label{eq:cp_cpree_mcp_comp}
\begin{split}
\{x \mid \tilde\xi_t(x)=0\} 
\supset \{x \mid \xi_t(x)\neq2\} 
\supset \{x \mid \eta_t(x)\neq2\}\\
\{x \mid \tilde\xi_t(x)=1\} 
\subset \{x \mid \xi_t(x)=2\} 
\subset \{x \mid \eta_t(x)=2\}
\end{split}
\end{align}
holds for $t=0$ then it holds for all $t$.
If $\overline\lambda(\beta,c,\alpha)>\lambda_c$ then the contact process $\tilde\xi_t$ can be made supercritical while preserving these subset relationships. 

For the standard contact process, let $\mathbb P^{\{x\}}_{\!\!_{\mathrm{cp}}}$ be the distribution of the process with $x$ being the only initially infected site and all other sites empty (state 0). For the CPREE and MCP let $x$ be occupied by type 2 initially and every other site be in state 1, and let $\mathbb P^{\{x\}}_{\!\!_{\mathrm{cpree}}}$ and $\mathbb P^{\{x\}}_{\!\!_{\mathrm{mcp}}}$ be the distributions of these processes respectively.  
That \eqref{eq:cp_cpree_mcp_comp} holds for all $t\geq0$ provides a stochastic ordering
relationship among these three measures which gives
\begin{equation*}
\mathbb P^{\{x\}}_{\!\!_{\mathrm{cp}}}(\tilde\xi_t(x)=1)\leq \mathbb P^{\{x\}}_{\!\!_{\mathrm{cpree}}}(\xi_t(x)=2)\leq \mathbb P^{\{x\}}_{\!\!_{\mathrm{mcp}}}(\eta_t(x)=2).
\end{equation*}
These processes, as graphically coupled here, satisfy that whenever $\tilde\xi_t(x)=1$ for the standard contact process we have that $\eta_t(x)=2$ for the MCP. This occurs when there is an active path connecting $x$ at time $0$ to itself at time $t$ using only unblocked 2-arrows. By our graphical coupling every active path for the standard contact process is an active path for the MCP.

Let $\mathcal A_1$ be the event that there are active paths from $(x,0)$ to $(x,t)$ for infinitely many $t$ for the standard contact process, and let $\mathcal A_2$ be the event that there are active paths from $(x,0)$ to $(x,t)$ for infinitely many $t$ for the MCP which only use unblocked 2-arrows. Since type 1 of the standard contact process survives strongly we have
\begin{equation}\label{eq:strongsurv2ineq}
0<\mathbb P^{\{x\}}_{\!\!_{\mathrm{cp}}}(x\in A_t \text{ i.o.})=\mathbb P^{\{x\}}_{\!\!_{\mathrm{cp}}}(\mathcal A_1)  \leq \mathbb P^{\{x\}}_{\!\!_{\mathrm{mcp}}}(\mathcal A_2) \leq \mathbb P^{\{x\}}_{\!\!_{\mathrm{mcp}}}(x\in B_t \text{ i.o.}).
\end{equation}
thus type 2 in the MCP survives strongly. 
\end{proof} 
The last inequality in \eqref{eq:strongsurv2ineq} is not necessarily equality since there may be active paths for the MCP that use blocked 2-arrows. Which arrows are used will depend on the particular realized graphical structure.

Solving $\overline\lambda=\frac2d\geq\lambda_c$ shows that $c>\frac{2}{\beta d}+\frac{4d\alpha}{d\alpha-2}$ is a sufficient condition for survival of type 2 when $\alpha>\frac2d$, $\beta>\frac2d$. For $\alpha$ large enough $c>5$ is eventually sufficient for any dimension. If $\alpha$ is small, then $c$ may be required to be quite large by the bound given here. 
For $d\geq 3$ improved upper bounds on the contact process critical value were recently derived in \cite{xue2018}.
With $\beta=4$ both $6=c<\alpha=8$ and $7=c>\alpha=6$ give $\overline\lambda>\frac2d\geq \lambda_c$ for any dimension. 
These are quite extreme parameter values and are far away from the threshold of the BDR conjecture, which claims that $c>1, c\beta>\lambda_c$ is all that is required. 

Here are three remarks that follow directly from our results and the stochastic comparisons used in this paper.
\begin{remark}
	The CPREE considered here survives strongly when $\overline\lambda>\lambda_c$.
	This is a direct result from the comparison between the CPREE and the standard contact process used in the proof of Theorem \ref{thm:cpree_mcp}.
\end{remark}
\begin{remark}
	Consider the MCP with parameters as in Theorem \ref{thm:cpree_mcp} with $\overline\lambda>\lambda_c$. If $c>\alpha>1$, then type 1 dies out in the sense that starting from a translation invariant initial distribution $\mu$ with $\mu(\eta(x)=i)>0$ for $i=1,2$, $\mathbb P^\mu(\eta_t(x)=1)\rightarrow 0$ as $t\rightarrow\infty$. %
	This is a direct result of Proposition \ref{lem:mcp1} 
	since $\beta_2>\beta_1$ and $\delta_2<\delta_1$. 
\end{remark}
\begin{remark}
	Consider the processes with parameters as in the proof of Theorem \ref{thm:cpree_mcp}. 
	With translation invariant initial distributions $0<\mu_1(\tilde\xi(x)=1)\leq\mu_2(\eta(x)=2)$ for the standard contact process and MCP, respectively, we have that 
\begin{equation*}
0<\overline\nu_1(\tilde\xi(x)=1)=\lim_{t\rightarrow\infty}\mathbb P^{\mu_1}(\tilde\xi_t(x)=1)\leq \lim_{t\rightarrow\infty}\mathbb P^{\mu_2}(\eta_t(x)=2)\leq\overline\nu_2(\eta(x)=2)
\end{equation*}
giving survival for type 2 when started from translation invariant initial distributions. Whether the last limit exists is unknown since we do not have a complete convergence theorem when $\delta_1\neq\delta_2$ (see \cite{mountford2019} for complete convergence when $\delta_1=\delta_2$). Furthermore we have no useful comparison for type 1 in our MCP and thus are unable to show that type 1 dies out with $\lim_{t\rightarrow\infty}\mathbb P^{\mu_2}(\eta_t(x)=1)=0$.
\end{remark}

\section*{Acknowledgments}
We thank the referee for suggestions that significantly improved the manuscript.

\providecommand{\bysame}{\leavevmode\hbox to3em{\hrulefill}\thinspace}

\end{document}